\documentclass[ams,12pt]{article}
\usepackage{geometry}
\usepackage{amsmath,amssymb,amsfonts,amsthm}
\usepackage{lmodern}
\usepackage{stmaryrd}%bigsqcup
\usepackage{mathrsfs}
\usepackage{graphicx} %graph
\usepackage{verbatim}
\usepackage{ulem}
\usepackage{indentfirst}
\usepackage{extarrows}
\usepackage{mathrsfs}
\usepackage{bm}
\usepackage [T1] {fontenc}
\usepackage{fancyhdr}
\newtheorem{Th}{\bf Theorem}[section]

\theoremstyle{definition} %\theoremstyle{example}
\newtheorem{remark}{\bf Remark}[section]

\usepackage{amssymb}
\newtheorem{Lem}[Th]{\bf Lemma}
\newtheorem{Pro}[Th]{\bf Proposition}

\numberwithin{equation}{section}
\usepackage{hyperref}

\usepackage{amsthm}
\begin{document}

\author{ZHAOJUN XING}
\title{\bf{A single-component regularity criterion and Inviscid limit of axially symmetric MHD-Boussinesq system}}

\date{}

\maketitle

\begin{abstract}
In this paper, we first give a critical BKM-type blow-up criterion that only involves the horizontal swirl component of the velocity for the inviscid axially symmetric MHD-Boussinesq system. Moreover, we consider the inviscid limit of the viscous MHD-Boussinesq system, and the convergence rate for the viscosity coefficient tending to zero is obtained.
\end{abstract}

\thispagestyle{plain}\pagestyle{fancy} \fancyhf{}
\fancyhead[C]{ZHAOJUN XING} % 在偶数页的中间显示作者姓名
\fancyhead[CO]{A single-component regularity criterion and Inviscid limit} % 在奇数页的中间显示章名
\fancyhead[L,RO]{\thepage} % 在偶数页的左侧，奇数页的右侧显示页码
\fancyfoot{\empty}
\let\thefootnote\relax
 \let\thefootnote\relax
\section{Introduction}
MHD-Boussinesq system models the convection of an incompressible flow, which is driven by the buoyancy effect of the thermal or density field and the Lorentz force generated by the fluid magnetic field. In addition, it is closely related to Rayleigh-B\'enard convection. This convection occurs in a horizontal layer of conductive fluid heated from below, with the effect of the magnetic field. For a more detailed physical background, interested readers are referred to \cite{18,20,21,25}. In the following, we present the 3D MHD-Boussinesq system:
\begin{equation}\label{eq1.1}
\left\{\begin{array}{l}
\partial_t u+u \cdot \nabla u+\mu\Delta u+\nabla p=h \cdot \nabla h+\rho e_3, \\[2mm]
\partial_t h+u \cdot \nabla h-h \cdot \nabla u-\nu \Delta h=0, \\[2mm]
\partial_t \rho+u \cdot \nabla \rho-\kappa \Delta \rho=0, \\[2mm]
\nabla \cdot u=\nabla \cdot h=0.
\end{array}\right.
\end{equation}
Here $u \in \mathbb{R}^3$ stands for the velocity and $h \in \mathbb{R}^3$ stands for the magnetic field. $p \in \mathbb{R}$ denotes the pressure. $\mu > 0$, $\nu>0$ and $\kappa>0$ denote the constant kinematic viscosity, magnetic diffusivity, and thermal diffusivity, respectively.

The MHD-Boussinesq system consists of a coupling between the Boussinesq equation and the magnetohydrodynamic equations. For the full MHD-Boussinesq system, there are some works concentrated on the global well-posedness of weak and strong solutions. See \cite{01,02} and references therein for 2D cases. In the 3D case, Larios-Pei \cite{14} proved the local well-posedness results in Sobolev space. Liu-Bian-Pu \cite{12} proved the global well-posedness of strong solutions with nonlinear damping terms in the momentum equations.  Li \cite{15} proved the regularity criterion, which only involves the horizontal swirl component of the vorticity field, to a class of three-dimensional axisymmetric MHD-Boussinesq system without magnetic impedance and thermal diffusivity. Later Li-Pan \cite{16} proved the related criterion that only applies the horizontal swirl component of the velocity field.

In recent years, Bian-Pu \cite{03} proved the global regularity of a family of axially symmetric large solutions to the MHD-Boussinesq system without magnetic resistivity and thermal diffusivity under the assumption that the support of the initial thermal fluctuation is away from the z-axis and its projection on to the z-axis is compact. Later, this result was improved by Pan \cite{26} by removing the assumption on the data of the thermal fluctuation.

Our first aim is to prove a single-component regularity criterion of the 3D axially symmetric inviscid MHD-Boussinesq system in Sobolev space $H^m$$(\forall m\geq 3)$. Setting $\mu=0$ of (\ref{eq1.1}), we can get
\begin{equation}\label{eq1.2}
\left\{\begin{array}{l}
\partial_t u+u \cdot \nabla u+\nabla p=h \cdot \nabla h+\rho e_3, \\[2mm]
\partial_t h+u \cdot \nabla h-h \cdot \nabla u-\nu \Delta h=0, \\[2mm]
\partial_t \rho+u \cdot \nabla \rho-\kappa \Delta \rho=0, \\[2mm]
\nabla \cdot u=\nabla \cdot h=0.
\end{array}\right.
\end{equation}
In the cylindrical coordinates $(r, \theta, z)$, i.e., for $x=\left(x_1, x_2, x_3\right) \in \mathbb{R}^3$,
$$
r=\sqrt{x_1^2+x_2^2}, \quad \theta=\arctan \frac{x_2}{x_1}, \quad z=x_3,
$$
a solution of (\ref{eq1.2}) is called an axially symmetric solution, if and only if
$$
\left\{\begin{array}{l}
u=u_r(t, r, z)\boldsymbol{e_r}+u_\theta(t, r, z) \boldsymbol{e_\theta}+u_z(t, r, z) \boldsymbol{e_z}, \\[2mm]
h=h_r(t, r, z) \boldsymbol{e_r}+h_\theta(t, r, z) \boldsymbol{e_\theta}+h_z(t, r, z) \boldsymbol{e_z}, \\[2mm]
\rho=\rho(t, r, z),
\end{array}\right.
$$
satisfy the system (\ref{eq1.2}). Here the basis vectors $\boldsymbol{e}_r, \boldsymbol{e}_\theta, \boldsymbol{e}_z$ are
$$
\boldsymbol{e}_r=\left(\frac{x_1}{r}, \frac{x_2}{r}, 0\right), \quad \boldsymbol{e}_\theta=\left(-\frac{x_2}{r}, \frac{x_1}{r}, 0\right), \quad \boldsymbol{e}_z=(0,0,1).
$$

From the local existence and uniqueness results, it is clear that one only needs to assume $h_0 \cdot \boldsymbol{e}_r=h_0 \cdot \boldsymbol{e}_z \equiv 0$, then vanishing of $h_r$ and $h_z$ holds for all time. In this case, Choosing $\nu=\kappa=1$ for simplicity, (\ref{eq1.2}) can be simplified to
\begin{equation}\label{eq1.3}
\left\{\begin{array}{l}
\partial_t u_r+\left(u_r \partial_r+u_z \partial_z\right) u_r+\partial_r P=-\frac{\left(h_\theta\right)^2}{r}+\frac{u_\theta^2}{r}, \\[2mm]
\partial_t u_\theta+\left(u_r \partial_r+u_z \partial_z\right) u_\theta=-\frac{u_\theta u_r}{r}, \\[2mm]
\partial_t u_z+\left(u_r \partial_r+u_z \partial_z\right) u_z+\partial_z P=\rho, \\[2mm]
\partial_t h_\theta+\left(u_r \partial_r+u_z \partial_z\right) h_\theta-\frac{h_\theta u_r}{r}=\left(\Delta-\frac{1}{r^2}\right) h_\theta, \\[2mm]
\partial_t \rho+\left(u_r \partial_r+u_z \partial_z\right) \rho-\Delta \rho=0, \\[2mm]
\nabla \cdot u=\partial_r u_r+\frac{u_r}{r}+\partial_z u_z=0 ,
\end{array}\right.
\end{equation}
where $P=p+\frac{1}{2}|h_\theta|^2$.

Denoting
$$
\Phi_{k, \alpha}(t)\leq c\underbrace{\exp(c\exp(\cdot\cdot\cdot \exp(ct^\alpha)))}_{k\ times\ exponents},\quad c>0, \quad k\geq 1.
$$

Our main result reads:
\begin{Th}\label{Th1.1}
Given $m \in \mathbb{N}$ and $m \geq 3$. Suppose $(u, h, \rho)$ is the unique strong solution of (\ref{eq1.2}) with initial data $\left(u_0, h_0, \rho_0\right) \in H^m \times H^m \times H^m, r \rho_0 \in L^2$ and $\nabla \cdot u_0=h_0 \cdot \boldsymbol{e}_r=h_0 \cdot \boldsymbol{e}_z \equiv 0$. Then, $(u, h, \rho)(t,\cdot)$ can be smoothly extended before $T_\ast$ if and only if a Beale-Kato-Majda-type condition on the swirl part of the velocity (\ref{eq1.4}) holds.
\begin{equation}\label{eq1.4}
\int_0^{T_\ast}\|\nabla\times (u_\theta \mathbf{e_{\theta}})(t,\cdot)\|_{L^\infty}dt\leq C_\ast<\infty.
\end{equation}
Now, $(u, h, \rho)(t,\cdot)$ satisfies the following temporal asymptotic property:
$$
\|(u,h,\rho)(t,\cdot)\|_{H^m}^2\leq \Phi_{4,3}(t),\ \ \ \ \forall t\leq T_\ast.
$$
\end{Th}
\noindent
\begin{remark}
When $u_\theta=0$, the global well-posedness result for the inviscid axisymmetric MHD-Boussinesq system can be found in \cite{17}. If $h=0$, \cite{15} gave a single-component Beale-Kato-Majda-type regularity criterion for inviscid Boussinesq equations. Our first result can be viewed as a generalization of the above papers.
\end{remark}

\qed

The problem of vanishing viscosity limit is one of the most challenging topics in fluid dynamics. For the inviscid limit problems on bounded domains, we refer readers to \cite{10,27} and \cite{06}. Since Kato \cite{10} and Swann \cite{27}, many authors have studied the convergence of the solution to the Navier-Stokes equation as the viscosity approaches zero. In the last decades, Itoh \cite{08} and Itoh Tani \cite{09} investigated the inviscid limit of the equation for non-homogeneous incompressible fluids and demonstrated the convergence in $H^2$ and $W^{1,p}$ $(p>3)$, respectively.
D\'iaz and Lerena \cite{07} investigated the inviscid and non-resistive limit in the Cauchy problem of an incompressible homogeneous MHD system by using the $C^0$-semigroup technique developed in \cite{11}.
Majda \cite{22} proved that when $\mu\rightarrow0$, the solution $u_\mu$ of the Navier-Stokes equation converges to the unique solution of the Euler equation in the $L^2$ norm and the convergence rate is $(\mu t)^{\frac{1}{2}}$, assuming $u_0\in H^s$, $s>\frac{d}{2}+2$. Later, Masmoudi \cite{23} improved this result and demonstrated the convergence in the $H^s$ norm under the weaker assumptions $u_0\in H^s$, $s>\frac{d}{2}+1$. Recently, Maafa and Zerguine \cite{24} studied the inviscid limit of MHD system in Besov spaces.

Our second goal is to examine the inviscid limit of the following 3D MHD-Boussinesq system with $\kappa=\nu=1$ of (\ref{eq1.1}):
\begin{equation}\label{eq1.5}
\left\{\begin{array}{l}
\partial_t u_\mu+u_\mu \cdot \nabla u_\mu-\mu\Delta u_\mu+\nabla p_\mu=h_\mu \cdot \nabla h_\mu+\rho_\mu e_3, \\[2mm]
\partial_t h_\mu+u_\mu \cdot \nabla h_\mu-h_\mu \cdot \nabla u_\mu-\Delta h_\mu=0, \\[2mm]
\partial_t \rho_\mu+u_\mu \cdot \nabla \rho_\mu-\Delta \rho_\mu=0, \\[2mm]
\nabla \cdot u_\mu=\nabla \cdot h_\mu=0.
\end{array}\right.
\end{equation}
Here, the solution is also supposed to be axially symmetric:
$$
\left\{\begin{array}{l}
u_\mu=u_{\mu,r}(t, r, z) \boldsymbol{e_r}+u_{\mu,\theta}(t, r, z) \boldsymbol{e_\theta}+u_{\mu,z}(t, r, z) \boldsymbol{e_z}, \\[2mm]
h_\mu=h_{\mu,r}(t, r, z) \boldsymbol{e_r}+h_{\mu,\theta}(t, r, z) \boldsymbol{e_\theta}+h_{\mu,z}(t, r, z) \boldsymbol{e_z}, \\[2mm]
\rho_\mu=\rho_\mu(t, r, z).
\end{array}\right.
$$
When the viscosity coefficient $\mu$ is vanishing, the MHD-Boussinesq system (\ref{eq1.5}) appears to degenerate into the system (\ref{eq1.2}). Our second main result quantified the rate of this convergence:
\begin{Th}\label{Th1.2}
Let $(u, h, \rho)$ and $(u_\mu, h_\mu, \rho_\mu)$ are strong solutions to the system (\ref{eq1.2}) and the system (\ref{eq1.5}) respectively with the same initial data. Then under the conditions as in Theorem \ref{Th1.1}, the following asymptotic behavior holds:
$$
\left\|u_\mu-u\right\|_{L_t^{\infty} L^2}+\left\|h_\mu-h\right\|_{L_t^{\infty}L^2}+\left\|\rho_\mu-\rho\right\|_{L_t^{\infty} L^2}\leq (\mu t)\Phi_{4,3}(t).
$$
\end{Th}
\noindent
%\textit{Remark 1.2.}
\begin{remark}
\cite{24} gave the inviscid limit of MHD system in Besov space, and in Theorem \ref{Th1.2} we use a similar method to give the inviscid limit of MHD-Boussinesq system in $L^2$.
\end{remark}

\qed

The detailed statements of Theorem \ref{Th1.1} and Theorem  \ref{Th1.2} can be found in Sections 3 and 4. We next introduce some notations, conventions and lemmas to be used in the proof.
\section{Preliminaries}
\textbf{Notations:} In this paper, we use $C_{a, b, c, \ldots}$ to represent a positive constant depending on $a, b, c, \ldots$ and it may be different from line to line. For $A \lesssim B$, it means $A \leq C B$. And $A \simeq B$ means both $A \lesssim B$ and $B \lesssim A$. $[\mathcal{A}, \mathcal{B}]=\mathcal{A B}-\mathcal{B} \mathcal{A}$ denotes the communicator of the operator $\mathcal{A}$ and the operator $\mathcal{B}$. We give the usual Lebesgue space $L^p$, Sobolev functional space $W^{k, p}$, and the usual homogeneous Sobolev space $\dot{W}^{k, p}$, respectively. When $p=2$, replace $W^{k, p}$ and $\dot{W}^{k, p}$ with $H^k$ and $\dot{H}^k$. $f \in L^p \cap L^q$ with $1 \leq p, q \leq \infty$, we shall denote its Yudovich type norm as $\|f\|_{L^p \cap L^q}=\max \left\{\|f\|_{L^p},\|f\|_{L^q}\right\}$. Given a Banach space $X$, we say $v:[0, T] \times \mathbb{R}^3 \rightarrow \mathbb{R}$ belongs to the Bochner-Banach space $L^p(0, T ; X)$, if $\|v(t, \cdot)\|_X \in L^p(0, T)$, and we usually use $L_T^p X$ for short notation of $L^p(0, T ; X)$.

In this paper, we do not distinguish functional spaces for scalar or vector-valued functions since it will be clear from the context.

Now we introduce some essential lemmas.
First of all, we give the Gagliardo-Nirenberg interpolation inequality, which we will not prove here.
\begin{Lem}\label{Lem2.1} (Gagliardo-Nirenberg). Fix $q, r \in[1, \infty]$ and $j, m \in \mathbb{N} \cup\{0\}$ with $j \leq m$. Suppose that $f \in L^q \cap \dot{W}^{m, r}(\mathbb{R}^d)$ and there exists a real number $\alpha \in[j / m, 1]$ such that
$$
\frac{1}{p}=\frac{j}{3}+\alpha\left(\frac{1}{r}-\frac{m}{3}\right)+\frac{1-\alpha}{q} .
$$
Then $f \in \dot{W}^{j, p}\left(\mathbb{R}^d\right)$ and there exists a constant $C>0$ such that
$$
\left\|\nabla^j f\right\|_{L^p} \leq C\left\|\nabla^m f\right\|_{L^r}^\alpha\|f\|_{L^q}^{1-\alpha},
$$
except the following two cases: \\
$(i)$ $j=0, m r<d$ and $q=\infty$; (In this case it is necessary to assume also that either $u \rightarrow 0$ at infinity, or $u \in L^s\left(\mathbb{R}^d\right)$ for some $s<\infty$.)\\
$(ii)$ $1<r<\infty$ and $m-j-3 / r \in \mathbb{N}$. (In this case it is necessary to assume also that $\alpha<1$.)\\
\end{Lem}
Next, we state the following estimation of the triple product form with commutators, the proof of which can be found in Lemma 3.8 of \cite{15}.
\begin{Lem}\label{Lem2.2} Let $m\in\mathbb{N}$ and $m\geq2$, $f,g,k\in C_0^\infty (\mathbb{R}^3)$. The following estimate holds:
$$
\left|\int_{\mathbb{R}^3}[\nabla^m,f\cdot\nabla]g\nabla^m kdx \right| \leq C\|\nabla^m(f,g,k)\|_{L^2}^2\|\nabla(f,g)\|_{L^\infty};
$$
\end{Lem}
The following inequality is about logarithmic Sobolev inequality.
\begin{Lem}\label{Lem2.3} For any divergence free vector field $g$ such that
$$
g: \mathbb{R}\rightarrow\mathbb{R}^3
$$
and $g\in H^3(\mathbb{R}^3)$, the following estimate holds:
$$
\|\nabla g\|_{L^\infty (\mathbb{R}^3)}\lesssim 1+\|\nabla\times g\|_{L^\infty}log \left(e+\|g\|_{H^3(\mathbb{R}^3)} \right).
$$
\end{Lem}
The following lemma can be obtained from the Biot-Savart law and the $L^p$ boundedness of the Calderon-Zygmund singular integral operator, and its detailed proof is referred to \cite{04,05}.
\begin{Lem}\label{Lem2.4} Let $u=u_r \boldsymbol{e_r}+u_\theta \boldsymbol{e_\theta}+u_z \boldsymbol{e_z}$ be an axially symmetric vector field, $\omega=\nabla \times u=\omega_r \boldsymbol{e_r}+$ $\omega_\theta \boldsymbol{e_\theta}+\omega_z \boldsymbol{e_z}$ and $b=u_r \boldsymbol{e_r}+u_z \boldsymbol{e_z}$. Then, we have
$$
\|\nabla u\|_{L^p} \leq C_p\|\omega\|_{L^p},
$$
and
\begin{equation}\label{eq2.1}
\|\nabla b\|_{L^p} \leq C_p\left\|\omega_\theta\right\|_{L^p},
\end{equation}
for all $1<p<\infty$.\\
\end{Lem}

Here's a famous lemma, and the details of the proof can be seen in equation (A.5) of \cite{13} and Proposition 2.5 of \cite{19}.
\begin{Lem}\label{Lem2.5} Define $\Omega:=\frac{w_\theta}{r}$. For $1<p<+\infty$, there exists an absolute constant $C_p>0$ such that
\begin{equation}\label{eq2.2}
\left\|\nabla \frac{u_r}{r}(t, \cdot)\right\|_{L^p} \leq C_p\|\Omega(t, \cdot)\|_{L^p} .
\end{equation}
\end{Lem}

Below we give the one-dimensional Hardy inequality,
\begin{Lem}\label{Lem2.6} If $p>1$, $\sigma\neq1$, $f$ is a nonnegative measurable function and $F$ is defined by
$$
F(x)=\int_0^x f(t)dt,\ \ (\sigma>1),\ \ F(x)=\int_0^x f(t)dt,\ \ (\sigma<1).
$$
Then,
$$
\int_0^\infty x^{-\sigma} F^p dx< (\frac{p}{|\sigma-1|})^p\int_0^\infty x^{-\sigma} (xp)^p dx,
$$
unless $f\equiv0$.\\
\end{Lem}
By Lemma \ref{Lem2.6}, we can get the following lemma:
\begin{Lem}\label{Lem2.7}
$$
\|\frac{u_\theta}{r}(t,\cdot)\|_{L^\infty}\leq \frac{1}{2}\|\omega_z(t,\cdot)\|_{L^\infty},
$$
for any $t>0$.
\end{Lem}
The proof refers to \cite{15}.
\section{Proof of Theorem \ref{Th1.1}}
\subsection{The main proof of Theorem \ref{Th1.1}}

This section is to prove Theorem \ref{Th1.1}. First, we derive the fundamental estimates of the system (\ref{eq1.2}) from the basic energy estimate. Secondly, we will define four special quantities to establish a reformulated system and derive a self-closed estimate. Then, we derive a one-component BKM-type criterion for the MHD-Boussinesq system. Finally, we conclude the proof of Theorem \ref{Th1.1}. In the following, we give the proof details.

First of all, we give the fundamental energy estimates of the system (\ref{eq1.2}).
\begin{Pro}\label{Pro3.1}(Fundamental Energy Estimates) Let $(u, h,\rho)$ be a smooth solution of (\ref{eq1.2}), then we have\\
(i) for $p\in[2,\infty)$ and $t\in \mathbb{R}_+$,
\begin{equation}\label{eq3.1}
\begin{aligned}
\|\mathcal{H}(t, \cdot)\|_{L^p}^p+\int_0^t \int_{\mathbb{R}^3}|\nabla \mathcal{H}(s, x)|^2|\mathcal{H}(s, x)|^{p-2} d x d s & \leq\left\|\mathcal{H}_0\right\|_{L^p}^p ;\\
\|\mathcal{H}(t,\cdot)\|_{L^\infty}\leq \|\mathcal{H}_0\|_{L^\infty};\\
\|\rho(t, \cdot)\|_{L^p}^p+\int_0^t \int_{\mathbb{R}^3}|\nabla \rho(s, x)|^2|\rho(s, x)|^{p-2} d x d s & \leq\left\|\rho_0\right\|_{L^p}^p ; \\
\|\rho(t,\cdot)\|_{L^\infty}\leq \|\rho_0\|_{L^\infty}.
\end{aligned}
\end{equation}
(ii) for $(u_0,h_0,\rho_0)\in L^2$ and $t\in \mathbb{R}_+$,
\begin{equation}\label{eq3.2}
\begin{aligned}
& \|(u, h)(t, \cdot)\|_{L^2}^2+\int_0^t\left\|\nabla h\left(s, \cdot\right)\right\|_{L^2}^2 d s \leq C_0(1+t)^2, \\
\end{aligned}
\end{equation}
where $C_0$ depends only on $\|(u_0,h_0,\rho_0)\|_{L^2}$.
\end{Pro}
\begin{proof}
Firstly, we define $\mathcal{H}:=\frac{h_\theta}{r}$, by $(\ref{eq1.3})_4$, one derives that
$$
\partial_t \mathcal{H}+\left(u_r \partial_r+u_z \partial_z\right) \mathcal{H}=\left(\Delta+\frac{2}{r} \partial_r\right) \mathcal{H} .
$$
Multiply the above equation by $p|\mathcal{H}|^{p-2}\mathcal{H}$ and integrate on $\mathbb{R}^3$, together with $\nabla u=0$, we can get:
$$
\frac{1}{p}\frac{d}{dt}\|\mathcal{H}(t,\cdot)\|_{L^p}^p+\int_{\mathbb{R}^3}|\nabla\mathcal{H}(t,\cdot)|^2|\mathcal{H}(t,\cdot)|^{p-2}dx\leq 0,
$$
and integrating over $(0, t)$ to obtain $(\ref{eq3.1})_{1,2}$.

Similarly for equation of $\rho$ of $(\ref{eq1.3})_5$, one derives $(\ref{eq3.1})_{3,4}$. For the estimation in (\ref{eq3.2}), we can get it by applying the standard $L^2$ inner product estimation for system (\ref{eq1.2}) and using $(\ref{eq3.1})_1$. Also, refer to Proposition 2.1 of \cite{26}.
\end{proof}

Next, we establish a reformulated system.

The vorticity of the axially symmetric velocity field $u$ is given by
$$
\omega(t,r,z)=\nabla\times u= \omega_r (t,r,z)\boldsymbol{e}_r+ \omega_\theta (t,r,z)\boldsymbol{e}_\theta+\omega_z (t,r,z)\boldsymbol{e}_z,
$$
where
$$
\omega_r =-\partial_z u_\theta, \quad \omega_\theta = \partial_z u_r -\partial_r u_z, \quad \omega_z = \partial_r u_\theta+\frac{u_\theta}{r}.
$$
By taking special derivatives of $(\ref{eq1.3})_{1,2,3}$, one concludes that $(\omega_r,\omega_\theta,\omega_z)$ satisfies
\begin{equation}\label{eq3.3}
\left\{\begin{array}{l}
\partial_t \omega_r+\left(u_r \partial_r+u_z \partial_z\right) \omega_r=\left(\omega_r \partial_r+\omega_z \partial_z\right) u_r, \\[2mm]
\partial_t \omega_\theta+\left(u_r \partial_r+u_z \partial_z\right) \omega_\theta=\frac{u_r}{r} \omega_\theta+\frac{1}{r} \partial_z\left(u_\theta^2\right)-\frac{1}{r} \partial_z\left(h_\theta^2\right)-\partial_r \rho,  \\[2mm]
\partial_t \omega_z+\left(u_r \partial_r+u_z \partial_z\right) \omega_z=\left(\omega_r \partial_r+\omega_z \partial_z\right) u_z. \\[2mm]
\end{array}\right.
\end{equation}
Denoting the following group of quantities:
$$
\begin{aligned}
\Omega:=\frac{\omega_\theta}{r},\quad J:=\frac{\omega_r}{r},\quad \mathcal{N}:=\frac{\partial_r\rho}{r},\quad \nabla\mathcal{H}:=\nabla\frac{h_\theta}{r}.
\end{aligned}
$$
From $(\ref{eq3.3})$ and Combined with $(\ref{eq1.3})_{4,5}$, we can get the following reformulated system for $(\Omega, J, \mathcal{N}, \nabla\mathcal{H})$:
\begin{equation}\label{eq3.4}
\begin{cases}
&\partial_t\Omega+u \cdot\nabla\Omega=-2\frac{u_\theta\omega_r}{r^2}-\partial_z \mathcal{H}^2-\mathcal{N},\\[2mm]
&\partial_t J+u \cdot\nabla J=(\omega_r\partial_r+\omega_z\partial_z)\frac{u_r}{r},\\[2mm]
&\partial_t \mathcal{N}+u \cdot \nabla \mathcal{N}-\left(\Delta+\frac{2}{r} \partial_r\right) \mathcal{N}=\partial_z u_z \mathcal{N}-\partial_r u_z \frac{\partial_z \rho}{r},\\[2mm]
&\partial_t \nabla \mathcal{H}+u \cdot \nabla \nabla \mathcal{H}+\nabla b \cdot \nabla \mathcal{H}-\left(\Delta+\frac{2}{r} \partial_r\right) \nabla \mathcal{H}-\nabla\left(\frac{2}{r}\right) \partial_r \mathcal{H}=0.
\end{cases}
\end{equation}

Then we give the following proposition of system (\ref{eq3.4}).
\begin{Pro}\label{Pro3.2} Let $(\Omega, J, \mathcal{N}, \nabla\mathcal{H})$ be defined as above, which solves $(\ref{eq3.4})$ with initial data
$$
\begin{aligned}
&(\Omega_0,J_0,\mathcal{N}_0,\mathcal{H}_0)\in (L^2\cap L^6)\times (L^2\cap L^6)\times L^2\times (L^\infty \cap H^1).
\end{aligned}
$$
Then, the following space-time estimate holds for any $t\in (0, T_\ast]$ that
\begin{equation}\label{eq3.5}
\begin{aligned}
\|(\Omega,J)(t,\cdot)\|^2_{L^2\cap L^6}+\|(\mathcal{N},\nabla\mathcal{H})(t,\cdot)\|^2_{L^2}+\int_0^t \|(\nabla\mathcal{N},\nabla^2\mathcal{H})(s,\cdot)\|^2_{L^2}ds\leq\Phi_{1,3}(t).
\end{aligned}
\end{equation}
\end{Pro}
\begin{proof} For the proof of (\ref{eq3.5}), we first derive an estimate about $\nabla\mathcal{H}$ from $\Omega$, and then an estimate about $\mathcal{N}$. Then, the estimate of $(\Omega,J)$ is obtained from the first two parts. Finally, we combine these estimates to get (\ref{eq3.5}). The following is specific to the process:\\

At the beginning, we derive the estimate of $\nabla\mathcal{H}$. By performing $L^2$ inner product of $\nabla\mathcal{H}$, $(\ref{eq3.4})_4$ follows that:
$$
\begin{aligned}
\frac{1}{2}\frac{d}{dt}\|\nabla\mathcal{H}(t,\cdot)\|_{L^2}^2+&\|\nabla^2\mathcal{H}(t,\cdot)\|_{L^2}^2-\int_{\mathbb{R}^3}\frac{2}{r}\partial_r\nabla \mathcal{N}\cdot\nabla \mathcal{N} dx-\sum^{2}_{i=1}\int_{\mathbb{R}^3}\partial_i(\frac{2}{r})\partial_r \mathcal{H}\partial_i \mathcal{H}dx\\
&=-\sum^{3}_{i,j=1}\int_{\mathbb{R}^3}\partial_i u_j\partial_j \mathcal{H}\partial_i \mathcal{H}dx.
\end{aligned}
$$
Here, one can follow the same way as Proposition 3.2 of \cite{17} to carry out estimates. This ends up with:
\begin{equation}\label{eq3.6}
\begin{aligned}
\|\nabla \mathcal{H}(t, \cdot)\|_{L^2}^2+\int_0^t\left\|\nabla^2 \mathcal{H}(s, \cdot)\right\|_{L^2}^2 d s \lesssim\left\|\nabla \mathcal{H}_0\right\|_{L^2}^2+\int_0^t\|\Omega(s, \cdot)\|_{L^2 \cap L^6}^2\|\nabla h(s, \cdot)\|_{L^2}^2 d s .
\end{aligned}
\end{equation}
Then, we get the estimate of $\mathcal{N}$. By taking the $L^2$ inner product with $\mathcal{N}$ for $(\ref{eq3.4})_3$, then integrating on $\mathbb{R}^3$, one has
\begin{equation}\label{eq3.7}
\begin{aligned}
\frac{1}{2} \frac{d}{d t}\|\mathcal{N}(t, \cdot)\|_{L^2}^2+\|\nabla \mathcal{N}(t, \cdot)\|_{L^2}^2&+2 \pi \int_0^{\infty}|\mathcal{N}(t, 0, z)|^2 d z\\[2mm]
&=\int_{\mathbb{R}^3} \partial_z u_z \mathcal{N}^2 d x-\int_{\mathbb{R}^3} \partial_r u_z \frac{\partial_z \rho}{r} \mathcal{N} d x .
\end{aligned}
\end{equation}
Using the method in the proof of Proposition 3.2 in \cite{15}, (\ref{eq3.7}) can be written by:
\begin{equation}\label{eq3.8}
\|\mathcal{N}(t,\cdot)\|_{L^2}^2+\int_0^t\|\nabla\mathcal{N}(s,\cdot)\|_{L^2}^2ds\lesssim\|\mathcal{N}_0\|_{L^2}^2+ \int_0^t\|\Omega(s,\cdot)\|_{L^2\cap L^6}^2\|\nabla\rho\|_{L^2}^2ds.
\end{equation}
Next is the estimate of $(\Omega, J)$. we perform $L^p$ $(2\leq p\leq 6)$ energy estimates of $(\ref{eq3.4})_1$ and $(\ref{eq3.4})_2$ respectively to get:
$$
\begin{aligned}
\frac{d}{dt}\|\Omega(t,\cdot)\|_{L^p}&\lesssim \|\frac{u_\theta}{r}(t,\cdot)\|_{L^\infty}\|J(t,\cdot)\|_{L^p}+\|\partial_z \mathcal{H}^2(t,\cdot)\|_{L^p}+\|\mathcal{N}(t,\cdot)\|_{L^p},\\[2mm]
\end{aligned}
$$
$$
\begin{aligned}
\frac{d}{dt}\|J(t,\cdot)\|_{L^p}&\lesssim \|(\omega_r,\omega_z)(t,\cdot)\|_{L^\infty}\|\nabla\frac {u_r}{r}(t,\cdot)\|_{L^p}.\\[2mm]
\end{aligned}
$$
By Lemma \ref{Lem2.7} and using the identity $\nabla\times(u_\theta \boldsymbol{e_\theta})= \omega_r \boldsymbol{e_r}+\omega_z \boldsymbol{e_z}$, together with (\ref{eq2.2}) of Lemma \ref{Lem2.5}, we can get:
$$
\begin{aligned}
\frac{d}{dt}\|\Omega(t,\cdot)\|_{L^p}&\lesssim \|\omega_z(t,\cdot)\|_{L^\infty}\|J(t,\cdot)\|_{L^p}+\|\partial_z \mathcal{H}^2(t,\cdot)\|_{L^p}+\|\mathcal{N}(t,\cdot)\|_{L^p}\\[2mm]
&\lesssim \|\nabla\times(u_\theta \boldsymbol{e_\theta})(t,\cdot)\|_{L^\infty}\|J(t,\cdot)\|_{L^p}+\|\partial_z \mathcal{H}^2(t,\cdot)\|_{L^p}+\|\mathcal{N}(t,\cdot)\|_{L^p},\\[2mm]
\end{aligned}
$$
$$
\begin{aligned}
\frac{d}{dt}\|J(t,\cdot)\|_{L^p}&\lesssim \|\nabla\times(u_\theta \boldsymbol{e_\theta})(t,\cdot)\|_{L^\infty}\|\nabla\frac{ u_r}{r}(t,\cdot)\|_{L^p}\\[2mm]
&\lesssim \|\nabla\times(u_\theta \boldsymbol{e_\theta})(t,\cdot)\|_{L^\infty}\|\Omega(t,\cdot)\|_{L^p}.\\[2mm]
\end{aligned}
$$
Integrating with time and using  $(\ref{eq3.1})_2$, one derives that:
\begin{equation}\label{eq3.9}
\begin{aligned}
\|\Omega(t,\cdot)\|_{L^p}\lesssim& \|\Omega_0\|_{L^p}+\int_0^t \|\nabla\times(u_\theta \boldsymbol{e_\theta})(s,\cdot)\|_{L^\infty}\|J(s,\cdot)\|_{L^p}ds\\
&+\|\mathcal{H}_0\|_{L^\infty}\int_0^t \|\partial_z \mathcal{H}(s,\cdot)\|_{L^p}ds+\int_0^t\|\mathcal{N}(s,\cdot)\|_{L^p}ds,
\end{aligned}
\end{equation}
\begin{equation}\label{eq3.10}
\begin{aligned}
\|J(t,\cdot)\|_{L^p}&\lesssim \|J_0\|_{L^p}+\int_0^t \|\nabla\times(u_\theta \boldsymbol{e_\theta})(s,\cdot)\|_{L^\infty}\|\Omega(s,\cdot)\|_{L^p}ds.
\end{aligned}
\end{equation}
Combined (\ref{eq3.9}) and (\ref{eq3.10}),
$$
\begin{aligned}
\|(\Omega,J)(t,\cdot)\|_{L^p}\lesssim& \|(\Omega_0,J_0)\|_{L^p}+\int_0^t \|\nabla\times(u_\theta \boldsymbol{e_\theta})(s,\cdot)\|_{L^\infty}\|(\Omega,J)(s,\cdot)\|_{L^p}ds\\[2mm]
&+\|\mathcal{H}_0\|_{L^\infty}\int_0^t \|\partial_z \mathcal{H}(s,\cdot)\|_{L^p}ds+\int_0^t\|\mathcal{N}(s,\cdot)\|_{L^p}ds.
\end{aligned}
$$
Using Gr\"{o}nwall inequality,
$$
\begin{aligned}
\|(\Omega,J)(t,\cdot)\|_{L^p}\lesssim \left(\|(\Omega_0,J_0)\|_{L^p}+\|\mathcal{H}_0\|_{L^\infty}\int_0^t \|\partial_z \mathcal{H}(s,\cdot)\|_{L^p}ds+\int_0^t\|\mathcal{N}(s,\cdot)\|_{L^p}ds\right)\\
\times\exp\left(\int_0^t \|\nabla\times(u_\theta \boldsymbol{e_\theta})(s,\cdot)\|_{L^\infty}ds\right).
\end{aligned}
$$
By (\ref{eq1.4}), for any $t\leq T_\ast$, we conclude for $2 \leq p\leq 6$ that:
\begin{equation}\label{eq3.11}
\begin{aligned}
&\|(\Omega,J)(t,\cdot)\|_{L^p}\\
\leq& C_{C_\ast} \left(\|(\Omega_0,J_0)\|_{L^p}+\|\mathcal{H}_0\|_{L^\infty}\int_0^t \|\partial_z \mathcal{H}(s,\cdot)\|_{L^p}ds+\int_0^t\|\mathcal{N}(s,\cdot)\|_{L^p}ds\right).\\
\end{aligned}
\end{equation}
Finally, choosing $p=2$ and $p=6$ in (\ref{eq3.11}), one derives
$$
\begin{aligned}
&\|\Omega(t, \cdot)\|_{L^2 \cap L^6}^2 \\
\lesssim & \left\|\Omega_0\right\|_{L^2 \cap L^6}^2+\left\|\mathcal{H}_0\right\|_{L^{\infty}}^2\left(\int_0^t\left\|\partial_z \mathcal{H}(s, \cdot)\right\|_{L^2} d s+\int_0^t\left\|\partial_z \mathcal{H}(s, \cdot)\right\|_{L^6}d s\right)^2 \\
& +\left(\int_0^t\|\mathcal{N}(s, \cdot)\|_{L^2} d s+\int_0^t\|\mathcal{N}(s, \cdot)\|_{L^6} d s\right)^2 .
\end{aligned}
$$
Using the Sobolev inequality and the H\"{o}lder's inequality, one deduces
$$
\begin{aligned}
\|\Omega(t, \cdot)\|_{L^2 \cap L^6}^2 \lesssim & \left\|\Omega_0\right\|_{L^2 \cap L^6}^2+t\left\|\mathcal{H}_0\right\|_{L^{\infty}}^2\left(\int_0^t\|\nabla \mathcal{H}(s, \cdot)\|_{L^2}^2 d s+\int_0^t\left\|\nabla^2 \mathcal{H}(s, \cdot)\right\|_{L^2}^2 d s\right) \\
& +\left(t^2 \sup _{s \in(0, t)}\|\mathcal{N}(s, \cdot)\|_{L^2}^2+t \int_0^t\|\nabla \mathcal{N}(s, \cdot)\|_{L^2}^2 d s\right) .
\end{aligned}
$$
Substituting (\ref{eq3.6}) and (\ref{eq3.8}) in the right hand side, we arrive
$$
\begin{aligned}
&\|\Omega(t, \cdot)\|_{L^2 \cap L^6}^2 \\
\lesssim &\left\|\Omega_0\right\|_{L^2 \cap L^6}^2+t\left\|\mathcal{H}_0\right\|_{L^{\infty}}^2\left(\left\|\mathcal{H}_0\right\|_{L^2}^2+\left\|\nabla \mathcal{H}_0\right\|_{L^2}^2+\int_0^t\|\Omega(s, \cdot)\|_{L^2 \cap L^6}^2\|\nabla h(s, \cdot)\|_{L^2}^2 d s\right) \\
& +\left(1+t^2\right)\left(\left\|\mathcal{N}_0\right\|_{L^2}^2+\int_0^t\|\Omega(s, \cdot)\|_{L^2 \cap L^6}^2\|\nabla \rho(s, \cdot)\|_{L^2}^2 d s\right) .
\end{aligned}
$$
This indicates, for any $t \leq T^\ast$,
$$
\begin{aligned}
&\|\Omega(t, \cdot)\|_{L^2 \cap L^6}^2 \\
\lesssim & \left\|\Omega_0\right\|_{L^2 \cap L^6}^2+t\left\|\mathcal{H}_0\right\|_{L^{\infty}}^2\left\|\mathcal{H}_0\right\|_{H^1}^2+t\left\|\mathcal{H}_0\right\|_{L^{\infty}}^2 \int_0^t\|\Omega(s, \cdot)\|_{L^2 \cap L^6}^2\|\nabla h(s, \cdot)\|_{L^2}^2 d s \\
& +\left(1+t^2\right)\left\|\mathcal{N}_0\right\|_{L^2}^2+\left(1+t^2\right) \int_0^t\|\Omega(s, \cdot)\|_{L^2 \cap L^6}^2\|\nabla \rho(s, \cdot)\|_{L^2}^2 d s .
\end{aligned}
$$
Thus by the Gr\"{o}nwall inequality:
$$
\begin{aligned}
&\|\Omega(t, \cdot)\|_{L^2 \cap L^6}^2 \\
\lesssim & \left(\left\|\Omega_0\right\|_{L^2 \cap L^6}^2+t\left\|\mathcal{H}_0\right\|_{L^{\infty}}^2\left\|\mathcal{H}_0\right\|_{H^1}^2+\left(1+t^2\right)\left\|\mathcal{N}_0\right\|_{L^2}^2\right) \\
& \times \exp \left(t\left\|\mathcal{H}_0\right\|_{L^{\infty}}^2 \int_0^t\|\nabla h(s, \cdot)\|_{L^2}^2 d s+\left(1+t^2\right) \int_0^t\|\nabla \rho(s, \cdot)\|_{L^2}^2 d s\right) .
\end{aligned}
$$
Using the fundamental energy estimates $(\ref{eq3.1})_1$ and $(\ref{eq3.1})_3$, one has
\begin{equation}\label{eq3.12}
\|\Omega(t, \cdot)\|_{L^2 \cap L^6}^2 \leq C_{0, C \ast}\left(1+t^2\right) \exp \left(C_0\left(1+t^3\right)\right) \leq \Phi_{1,3}(t), \quad \forall t \in\left(0, T_\ast\right] .
\end{equation}
Substituting (\ref{eq3.12}) in (\ref{eq3.6}) and (\ref{eq3.8}) respectively, using $(\ref{eq3.1})_1$ and $(\ref{eq3.1})_3$, one concludes
\begin{equation}\label{eq3.13}
\begin{aligned}
& \|\mathcal{N}(t, \cdot)\|_{L^2}^2+\int_0^t\|\nabla \mathcal{N}(s, \cdot)\|_{L^2}^2 d s+\|\nabla \mathcal{H}(t, \cdot)\|_{L^2}^2+\int_0^t\left\|\nabla^2 \mathcal{H}(s, \cdot)\right\|_{L^2}^2 d s \\
\leq & \Phi_{1,3}(t) \int_0^t\left(\|\nabla \rho(s, \cdot)\|_{L^2}^2+\|\nabla h(s, \cdot)\|_{L^2}^2\right) d s \\
\leq & \Phi_{1,3}(t), \quad \forall t \in\left(0, T_\ast\right] .
\end{aligned}
\end{equation}
Thus the proposition is proved by combining (\ref{eq3.12}) and (\ref{eq3.13}).
\end{proof}

The purpose of this next part is we derive one-component BKM-type criterion for MHD-Boussinesq system in this order: we first derive the $\|r\rho\|_{L^\infty_t L^2\cap L^2_t\dot{H}^1}$ and then get the $\|\nabla\rho\|_{L^\infty_t L^2\cap L^2_t\dot{H}^1}$. Next, we derive the $\|\nabla u\|_{L^\infty_t(L^2 \cap L^6)}$. It then follows that deduce the $\|(\nabla\partial_z\mathcal{H},\nabla h,\nabla^2 h,\nabla^2 \rho)\|_{L^\infty_t L^2\cap L^2_t\dot{H}^1}$. Finally, we deduce the $\|(\omega_\theta,\nabla h,\nabla \rho)\|_{L^1_t L^\infty }$.

We first give $L^{\infty}_t L_t^2\cap L_t^2\dot{H}^1$ estimate of $r\rho$ and $\nabla \rho$.
\begin{Pro}\label{Pro3.3} Under the same conditions as Theorem \ref{Th1.1}, $r\rho$ satisfies the following space-time estimate
\begin{equation}\label{eq3.14}
\|r \rho(t, \cdot)\|_{L^2}^2+\int_0^t\|\nabla(r \rho)(s, \cdot)\|_{L^2}^2 d s \leq C_0(1+t)^3,
\end{equation}
where $C_0>0$ is a constant depending only on the initial data $u_0, h_0$, and $\rho_0$.
\end{Pro}

Next, we will use the weighted estimate of $r \rho$ of (\ref{eq3.14}) in Proposition \ref{Pro3.3} to establish the $L_t^{\infty} L^2 \cap L_t^2 \dot{H}^1$ estimate of $\nabla \rho$. This proposition can be found in Proposition 3.4 of \cite{15}, so we omit proof here.
\begin{Pro}\label{Pro3.4} Under the same conditions as Theorem \ref{Th1.1}, $\nabla \rho$ satisfies the following space-time estimate
\begin{equation}\label{eq3.15}
\|\nabla \rho(t, \cdot)\|_{L^2}^2+\int_0^t\left\|\nabla^2 \rho(s, \cdot)\right\|_{L^2}^2 d s \leq \Phi_{1,3}(t).
\end{equation}
\end{Pro}
Before listing the estimate of the critical proof of the velocity field, we give a proposition to use.
\begin{Pro}\label{Pro3.5}
Under the same conditions as Theorem \ref{Th1.1}, the following $L^p$ estimate of $h_\theta$ satisfies
\begin{equation}\label{eq3.16}
\left\|h_\theta(t, \cdot)\right\|_{L^p} \leq \Phi_{2,3}(t),
\end{equation}
where for $p\in[2, \infty)$ is uniform.
\end{Pro}
\begin{proof}
For any $p \geq 2$, taking $L^p$ inner product of $h_{\theta}$ on $(\ref{eq1.3})_4$, one derives
$$
\begin{aligned}
&\frac{1}{p} \frac{d}{d t}\left\|h_\theta(t, \cdot)\right\|_{L^p}^p \\
\leq&\left\|\frac{u_r}{r}(t, \cdot)\right\|_{L^{\infty}}\left\|h_\theta(t, \cdot)\right\|_{L^p}^p-\int_{\mathbb{R}^3} \frac{\left|h_\theta\right|^p}{r^2} d x-(p-1) \int_{\mathbb{R}^3}\left|\nabla h_\theta\right|^2\left|h_\theta\right|^{p-2} d x \\
\leq&\left\|\frac{u_r}{r}(t, \cdot)\right\|_{L^{\infty}}\left\|h_\theta(t, \cdot)\right\|_{L^p}^p .
\end{aligned}
$$
Here using Lemma \ref{Lem2.1}, Lemma \ref{Lem2.5}, together with (\ref{eq3.12}), one derives that:
\begin{equation}\label{eq3.17}
\begin{aligned}
\|\frac{u_r}{r}(t,\cdot)\|_{L^\infty}&\lesssim \|\frac{u_r}{r}(t,\cdot)\|_{L^6}^{\frac{1}{2}}\|\nabla\frac{u_r}{r}(t,\cdot)\|_{L^6}^{\frac{1}{2}}\lesssim \|\nabla\frac{u_r}{r}(t,\cdot)\|_{L^2}^{\frac{1}{2}}\|\nabla\frac{u_r}{r}(t,\cdot)\|_{L^6}^{\frac{1}{2}}\\
&\lesssim \|\Omega(t,\cdot)\|_{L^2}^{\frac{1}{2}}\|\Omega(t,\cdot)\|_{L^6}^{\frac{1}{2}}
\leq \Phi_{1,3}(t),
\end{aligned}
\end{equation}
then using the Gr\"onwall inequality, one finds
$$
\begin{aligned}
&\left\|h_\theta(t, \cdot)\right\|_{L^p} \leq\left\|h_0 \cdot \boldsymbol{e_\theta}\right\|_{L^p} \exp \left(\int_0^t\left\|\frac{u_r}{r}(s, \cdot)\right\|_{L^{\infty}} d s\right) \leq \Phi_{2,3}(t), \\
&\quad \text { uniformly for } p \in[2, \infty) .
\end{aligned}
$$
\end{proof}

 Based on Proposition \ref{Pro3.2}, Proposition \ref{Pro3.4}, and Proposition \ref{Pro3.5}, we can now get the estimate of the velocity field.
\begin{Pro}\label{Pro3.6} Under the same conditions as Theorem \ref{Th1.1}, the following $L^2 \cap L^6$ estimate of $\nabla u$
$$
\|\nabla u(t,\cdot)\|_{L^2 \cap L^6}\leq \Phi_{2,3}(t),
$$
holds uniformly for $0\leq t\leq T_\ast$.
\end{Pro}
\begin{proof}
According to $(\ref{eq1.3})_2, \frac{u_\theta}{r}$ satisfies
$$
\partial_t \frac{u_\theta}{r}+(u \cdot \nabla) \frac{u_\theta}{r}+2 \frac{u_r}{r} \cdot \frac{u_\theta}{r}=0 .
$$
Multiplying $|\frac{u_\theta}{r}|^{p-2}\frac{u_\theta}{r}$ and integrating over $\mathbb{R}^3$, one derives
$$
\frac{d}{dt}\left\|\frac{u_\theta}{r}(t, \cdot)\right\|_{L^P} \leq 2 \left\|\frac{u_r}{r}(t, \cdot)\right\|_{L^{\infty}}\left\|\frac{u_{\theta}}{r}(t, \cdot)\right\|_{L^p}.
$$
Using Gr\"onwall inequality and (\ref{eq3.17}),
$$
\left\|\frac{u_\theta}{r}(t, \cdot)\right\|_{L^p} \leq\|\frac{u_0}{r}\cdot \boldsymbol{e_\theta}\|_{L^p} \exp \left(2 \int_0^t\left\|\frac{u_r}{r}(s, \cdot)\right\|_{L^{\infty}} d s\right) \leq \Phi_{2,3}(t), \quad \text { for any } \;p \in[2, \infty).
$$
Thus, we can get the $L^p$ estimate of $(\ref{eq3.3})_2$:
$$
\begin{aligned}
\|\omega_\theta(t,\cdot)\|_{L^p}\lesssim& \|\omega_0\cdot \boldsymbol{e_\theta}\|_{L^p}+\|\omega_r\|_{L^1_t L^\infty}\|\frac{u_\theta}{r}\|_{L^\infty_t L^p}+\|h_\theta\|_{L^\infty_t L^\infty}\|\partial_z \mathcal{H}\|_{L^1_t L^p}+\|\nabla\rho\|_{L^1_t L^p}\\
&+\int_0^t \|\omega_\theta(s,\cdot)\|_{L^p}\|\frac{u_r}{r}(s,\cdot)\|_{L^\infty}ds.\\
\end{aligned}
$$
Bying Gr\"onwall inequality and Sobolev inequality, together with Proposition \ref{Pro3.1}, (\ref{eq3.13}), (\ref{eq3.15}) and (\ref{eq3.16}), it follows that
$$
\begin{aligned}
\|\omega_\theta(t,\cdot)\|_{L^2}\leq& \left(\|\omega_0\cdot \boldsymbol{e_\theta}\|_{L^2}+\|\omega_r\|_{L^1_t L^\infty}\|\frac{u_\theta}{r}\|_{L^\infty_t L^2}+\|h_\theta\|_{L^\infty_t L^\infty}\|\nabla \mathcal{H}\|_{L^1_t L^2}+\|\nabla\rho\|_{L^1_t L^2}\right)\\
&\times\exp \left(\int_0^t \|\frac{u_r}{r}(s,\cdot)\|_{L^\infty}ds \right)\\
\leq& [1+\Phi_{2,3}(t)+\sqrt{t}\Phi_{2,3}(t)+\sqrt{t}]\Phi_{2,3}(t)\leq \Phi_{2,3}(t);\\
\|\omega_\theta(t,\cdot)\|_{L^6}\leq& \left(\|\omega_0\cdot \boldsymbol{e_\theta}\|_{L^6}+\|\omega_r\|_{L^1_t L^\infty}\|\frac{u_\theta}{r}\|_{L^\infty_t L^6}+\|h_\theta\|_{L^\infty_t L^\infty}\|\nabla \mathcal{H}\|_{L^1_t L^6}+\|\nabla \rho\|_{L^1_t L^6}\right)\\
&\times\exp \left(\int_0^t \|\frac{u_r}{r}(s,\cdot)\|_{L^\infty}ds \right)\\
\leq& [1+\Phi_{2,3}(t)+\Phi_{1,3}(t)\Phi_{2,3}(t)+\Phi_{1,3}(t)]\Phi_{2,3}(t)\leq \Phi_{2,3}(t),\\
\end{aligned}
$$
for any $t \leq T_\ast$. Then from (\ref{eq2.1}) in Lemma \ref{Lem2.4}, we have
$$
\begin{aligned}
\|\nabla b(t,\cdot)\|_{L^2\cap L^6}\leq \Phi_{2,3}(t).
\end{aligned}
$$

Below, we also need to get an estimate of about $\nabla\left(u_\theta \boldsymbol{e}_{\boldsymbol{\theta}}\right)$. Due to $\nabla \times$ $\left(u_\theta \boldsymbol{e}_{\boldsymbol{\theta}}\right)=w_r \boldsymbol{e}_{\boldsymbol{r}}+w_z \boldsymbol{e}_{\boldsymbol{z}}$ and $\operatorname{div}\left(u_\theta \boldsymbol{e}_{\boldsymbol{\theta}}\right)=0$ and using the argument of the Calderon-Zygmund singular integral operator, it only needs to prove the same estimate for $\left(w_r, w_z\right)$. Therefore, performing the $L^p$ estimates for $(\ref{eq3.3})_1$ and $(\ref{eq3.3})_3$, one derives
$$
\begin{aligned}
&\max \left\{\frac{\mathrm{d}}{\mathrm{d} t}\left\|w_r(t, \cdot)\right\|_{L^p}^p, \frac{\mathrm{d}}{\mathrm{d} t}\left\|w_z(t, \cdot)\right\|_{L^p}^p\right\} \\
\lesssim &  \int_{\mathbb{R}^3}\left(\left|w_r\right|^p+\left|w_z\right|^p\right)|\nabla b| \mathrm{d} x \\
\lesssim& \left\|\nabla \times\left(u_\theta \boldsymbol{e}_{\boldsymbol{\theta}}\right)(t, \cdot)\right\|_{L^{\infty}} \int_{\mathbb{R}^3}\left(\left|w_r\right|^{p-1}+\left|w_z\right|^{p-1}\right)|\nabla b| \mathrm{d} x.
\end{aligned}
$$
Using H\"older's inequality, it follows that
$$
\begin{aligned}
&\max \left\{\frac{\mathrm{d}}{\mathrm{d} t}\left\|w_r(t, \cdot)\right\|_{L^p}^p, \frac{\mathrm{d}}{\mathrm{d} t}\left\|w_z(t, \cdot)\right\|_{L^p}^p\right\} \\
\lesssim&\left\|\nabla \times\left(u_\theta \boldsymbol{e}_{\boldsymbol{\theta}}\right)(t, \cdot)\right\|_{L^{\infty}}\left(\left\|w_r(t, \cdot)\right\|_{L^p}^{p-1}+\left\|w_z(t, \cdot)\right\|_{L^p}^{p-1}\right)\\
&\times\|\nabla b(t, \cdot)\|_{L^p}.
\end{aligned}
$$
By dividing $\left(\left\|w_r(t, \cdot)\right\|_{L^p}^{p-1}+\left\|w_z(t, \cdot)\right\|_{L^p}^{p-1}\right)$ on both sides, it indicates that
$$
\max \left\{\frac{\mathrm{d}}{\mathrm{d} t}\left\|w_r(t, \cdot)\right\|_{L^p}, \frac{\mathrm{d}}{\mathrm{d} t}\left\|w_z(t, \cdot)\right\|_{L^p}\right\} \lesssim\left\|\nabla \times\left(u_\theta \boldsymbol{e}_{\boldsymbol{\theta}}\right)(t, \cdot)\right\|_{L^{\infty}}\|\nabla b(t, \cdot)\|_{L^p}.
$$
Integrating with $t$, for any $p \in[2,6]$, one concludes that:
$$
\begin{aligned}
\left\|\left(w_r, w_z\right)(t, \cdot)\right\|_{L^p}& \lesssim\left\|\left(w_0\cdot \boldsymbol{e_r},w_0\cdot \boldsymbol{e_z}\right)\right\|_{L^p}+\|\nabla b\|_{L_t^{\infty} L^p} \int_0^t\left\|\nabla \times\left(u_\theta \boldsymbol{e}_{\boldsymbol{\theta}}\right)(s, \cdot)\right\|_{L^{\infty}} \mathrm{d} s \\
&\leq \Phi_{2,3}(t).
\end{aligned}
$$
So, we complete the proof of Proposition \ref{Pro3.6}.
\end{proof}

Next, we derive the $L_t^1L^{\infty}$ estimate for the vector field $(\nabla\times u,\nabla\times h,\nabla\rho)$, which is the key to obtaining a higher-order estimation of the solution. And before we do that, we need to get an estimate of $\nabla\partial_z\mathcal{H},\nabla h,\nabla^2 h$ and $\nabla^2 \rho$.
\begin{Pro}\label{Pro3.7} Under the same conditions as Theorem \ref{Th1.1}, the following space-time estimate of $\nabla\partial_z\mathcal{H},\nabla h,\nabla^2 h$ and $\nabla^2 \rho$ holds:
$$
\begin{aligned}
\|(\nabla\partial_z\mathcal{H},\nabla h,\nabla^2 h,\nabla^2 \rho)(t,\cdot)\|_{L^2}^2+\int_0^t\|\nabla(\nabla\partial_z\mathcal{H},\nabla h,\nabla^2 h,\nabla^2 \rho)(s,\cdot)\|_{L^2}^2 \leq \Phi_{2,3}(t).
\end{aligned}
$$
\end{Pro}
\begin{proof} Firstly, we can apply $\partial_z$ on $(\ref{eq3.4})_4$ and perform the $L^2$ inner product to handle the $\nabla\partial_z \mathcal{H}$. Then we can get:
\begin{equation}\label{eq3.18}
\begin{aligned}
\|\nabla\partial_z \mathcal{H}(t,\cdot)\|_{L^2}^2+\int_0^t\|\nabla^2\partial_z\mathcal{H}(s,\cdot)\|_{L^2}^2 ds\leq \Phi_{2,3}(t).
\end{aligned}
\end{equation}
Secondly, we deal with $\nabla h$ and $\nabla^2 h$. Taking $\nabla$ and $\nabla^2$ on $(\ref{eq1.2})_2$ and performing the $L^2$ inner product respectively, we can conclude that:
\begin{equation}\label{eq3.19}
\begin{aligned}
\|\nabla h (t,\cdot)\|_{L^2}^2+\int_0^t\|\nabla^2 h(s,\cdot)\|_{L^2}^2 ds\leq \Phi_{2,3}(t),\\
\|\nabla^2 h (t,\cdot)\|_{L^2}^2+\int_0^t\|\nabla^3 h(s,\cdot)\|_{L^2}^2 ds\leq \Phi_{2,3}(t).
\end{aligned}
\end{equation}
Finally, applying $\nabla^2$ on $(\ref{eq1.3})_5$ and performing the $L^2$ inner product, we derive the estimate for $\rho$:
\begin{equation}\label{eq3.20}
\begin{aligned}
\|\nabla^2 \rho (t,\cdot)\|_{L^2}^2+\int_0^t\|\nabla^3 \rho(s,\cdot)\|_{L^2}^2 ds\leq \Phi_{2,3}(t).
\end{aligned}
\end{equation}
Proofs of the above estimates can be found in \cite{17}. By combining (\ref{eq3.18}), (\ref{eq3.19}) and (\ref{eq3.20}), we complete the proof of Proposition \ref{Pro3.7}.
\end{proof}

We give the $L^1_t L^\infty $ estimate of $\nabla\times u$, $\nabla\times h$, and $\nabla \rho$ in the following.
\begin{Pro}\label{Pro3.8} Under the same conditions as Theorem \ref{Th1.1}, the following $L^1_t L^\infty$ estimates of $\nabla \times u, \nabla \times h$ and $\nabla \rho$ follows
$$
\int_0^t\|(\omega_\theta, \nabla h, \nabla\rho)(s,\cdot)\|_{L^\infty} ds \leq \Phi_{2,3}(t).
$$
\end{Pro}
\begin{proof} Due to the equation of $w_\theta$ :
$$
\partial_t w_\theta+\left(u_r \partial_r+u_z \partial_z\right) w_\theta=\frac{u_r}{r} w_\theta+\frac{1}{r} \partial_z\left(u_\theta\right)^2-\frac{1}{r} \partial_z\left(h_\theta\right)^2-\partial_r \rho .
$$
Integrating it along the particle trajectory start at $x \in \mathbb{R}^3$, one knows that
$$
w_\theta(t, X(t, x))=\left(w_0 \cdot \boldsymbol{e}_\theta\right)(x)+\int_0^t\left(\frac{u_r}{r} w_\theta+\frac{1}{r} \partial_z\left(u_\theta\right)^2-\frac{1}{r} \partial_z\left(h_\theta\right)^2-\partial_r \rho\right)(s, X(s, x)) d s .
$$
Taking the $L^{\infty}$ norm with $x \in \mathbb{R}^3$, one derives from the previous estimates:
$$
\begin{aligned}
&\left\|w_\theta(t, \cdot)\right\|_{L^{\infty}} \\
\lesssim&\left\|w_0 \cdot \boldsymbol{e}_\theta\right\|_{L^{\infty}}+\int_0^t\left\|\frac{u_r}{r}(s, \cdot)\right\|_{L^{\infty}}\left\|w_\theta(s, \cdot)\right\|_{L^{\infty}} d s+\left\|\omega_r\right\|_{L^1\left(0, t, L^{\infty}\right)}\left\|\frac{u_\theta}{r}\right\|_{L^{\infty}\left(0, t, L^{\infty}\right)}\\
&+\left\|\partial_z \mathcal{H}\right\|_{L^1\left(0, t, L^{\infty}\right)}\left\|h_\theta\right\|_{L^{\infty}\left(0, t, L^{\infty}\right)}+\|\nabla \rho\|_{L^1\left(0, t, H^2\right)} \\[2mm]
& \leq \Phi_{2,3}(t)+\int_0^t\left\|\frac{u_r}{r}(s, \cdot)\right\|_{L^{\infty}}\left\|w_\theta(s, \cdot)\right\|_{L^{\infty}} d s .
\end{aligned}
$$
Using Gr\"onwall inequality, it indicates that
\begin{equation}\label{eq3.21}
\left\|w_\theta(t, \cdot)\right\|_{L^{\infty}} \leq \Phi_{2,3}(t) \exp \left(\int_0^t\left\|\frac{u_r}{r}(s, \cdot)\right\|_{L^{\infty}} d s\right)\leq\Phi_{2,3}(t).
\end{equation}
By Lemma \ref{Lem2.1} and H\"older's inequality, together with (\ref{eq3.19}), one derives
\begin{equation}\label{eq3.22}
\begin{aligned}
\int_0^t\|\nabla h(s, \cdot)\|_{L^{\infty}} d s & \lesssim \int_0^t\|\nabla h(s, \cdot)\|_{L^2}^{\frac{1}{4}}\left\|\nabla^3 h(s, \cdot)\right\|_{L^2}^{\frac{3}{4}} d s \\
& \lesssim\|\nabla h\|_{L^{\infty}\left(0, t, L^2\right)}\left(\int_0^t\left\|\nabla^3 h(s, \cdot)\right\|_{L^2}^2 d s\right)^{\frac{3}{8}} t^{\frac{5}{8}} \\
& \leq \Phi_{2,3}(t)\left(\Phi_{2,3}(t)\right)^{\frac{3}{8}} t^{\frac{5}{8}} \leq \Phi_{2,3}(t) .
\end{aligned}
\end{equation}
Similarly, using the estimates (\ref{eq3.15}), we get
\begin{equation}\label{eq3.23}
\int_0^t\|\nabla \rho(s, \cdot)\|_{L^{\infty}} d s \leq \Phi_{2,3}(t) .
\end{equation}
Combining (\ref{eq3.21}), (\ref{eq3.22}) and (\ref{eq3.23}), we complete Proposition \ref{Pro3.8} .
\end{proof}

\subsection{The end of the proof of Theorem \ref{Th1.1}} In this part, we will use the estimators obtained above to reach the conclusion of Theorem \ref{Th1.1}.

Applying $\nabla^m (m\in\mathbb{N},m\geq3)$ to $(\ref{eq1.2})_{1,2,3}$ and performing the $L^2$ energy estimate. we can obtain
\begin{equation}\label{eq3.24}
\begin{aligned}
& \frac{1}{2} \frac{d}{d t}\left\|\nabla^m(u, h, \rho)(t, \cdot)\right\|_{L^2}^2+\left\|\nabla^{m+1} h(t, \cdot)\right\|_{L^2}^2+\left\|\nabla^{m+1} \rho(t, \cdot)\right\|_{L^2}^2 \\
= & -\underbrace{\int_{\mathbb{R}^3}\left[\nabla^m, u \cdot \nabla\right] u \nabla^m u d x}_{I_1}+\underbrace{\int_{\mathbb{R}^3}\left[\nabla^m, h \cdot \nabla\right] h \nabla^m u d x}_{I_2}-\underbrace{\int_{\mathbb{R}^3}\left[\nabla^m, u \cdot \nabla\right] h \nabla^m h d x}_{I_3} \\
& +\underbrace{\int_{\mathbb{R}^3}\left[\nabla^m, h \cdot \nabla\right] u \nabla^m h d x}_{I_4}-\underbrace{\int_{\mathbb{R}^3}\left[\nabla^m, u \cdot \nabla\right] \rho \nabla^m \rho d x}_{I_5}+\underbrace{\int_{\mathbb{R}^3} \nabla^m \rho \nabla^m u d x}_{I_6} ,
\end{aligned}
\end{equation}
where we have used
$$
\int_{\mathbb{R}^3}h\cdot\nabla\nabla^m h\cdot\nabla^m udx+\int_{\mathbb{R}^3}h\cdot\nabla\nabla^m u\cdot\nabla^m hdx=0.
$$

By Lemma \ref{Lem2.2}, we have $I_1-I_5$ satisfy
\begin{equation}\label{eq3.25}
I_j \lesssim\left\|\nabla^m(u, h, \rho)(t, \cdot)\right\|_{L^2}^2\|\nabla(u, h, \rho)(t, \cdot)\|_{L^{\infty}}, \quad \forall j=1,2,3,4,5 ,
\end{equation}
and $I_6$ satisfies
\begin{equation}\label{eq3.26}
I_6 \leq\left\|\nabla^m \rho(t, \cdot)\right\|_{L^2}\left\|\nabla^m u(t, \cdot)\right\|_{L^2} \leq\left\|\nabla^m(u, h, \rho)(t, \cdot)\right\|_{L^2}^2 .
\end{equation}
Bring (\ref{eq3.25}) and (\ref{eq3.26}) into (\ref{eq3.24}), we can get
\begin{equation}\label{eq3.27}
\frac{d}{d t}\|(u, h, \rho)(t, \cdot)\|_{H^m}^2 \leq C\left(1+\|\nabla(u, h, \rho)(t, \cdot)\|_{L^{\infty}}\right)\|(u, h, \rho)(t, \cdot)\|_{H^m}^2 .
\end{equation}
Denoting
$$
E_m(t):=\|(u,h,\rho)(t,\cdot)\|_{H^m}^2,\ \ \ \forall t\leq T_\ast.
$$
By Lemma \ref{Lem2.3}, (\ref{eq3.27}) can be written
$$
E_m^{\prime}(t)\lesssim\left(1+\|(\nabla\times u,\nabla\times h,\nabla\rho)(t,\cdot)\|_{L^\infty}\log(e+E_m(t))\right)(e+E_m(t)).
$$
Using the Gr\"onwall inequality twice, one arrives at
\begin{equation}\label{eq3.28}
e+E_m(t)\leq (e+E_m(0))^{\exp\left(C\int_{0}^t(1+\|(\nabla\times u,\nabla h,\nabla\rho)(s,\cdot)\|_{L^\infty}ds\right)},\ \ \ \forall t\leq T_\ast.
\end{equation}
Recalling Proposition \ref{Pro3.8}, one has
$$
\int_0^t\left(1+\|(\nabla \times u, \nabla h, \nabla \rho)(s, \cdot)\|_{L^{\infty}}\right) d s \leq \Phi_{2,3}(t) .
$$
Substituting in (\ref{eq3.28}), one concludes that
$$
\sup _{0 \leq s \leq t}\|(u, h, \rho)(s, \cdot)\|_{H^m}^2 \leq \Phi_{4,3}(t),
$$
for all $m \in \mathbb{N}$. This completes the proof of Theorem \ref{Th1.1}.

\qed

\section{Proof of Theorem \ref{Th1.2}}
This section is devoted to study the inviscid limit of the viscous system \eqref{eq1.5}. Let $(u_\mu, h_\mu, \rho_\mu)$ and $(u, h, \rho)$ be a solution of (\ref{eq1.5}) and (\ref{eq1.2}), respectively. We also denote
$$
\bar{u}_\mu=u_\mu-u,\quad \bar{h}_\mu = h_\mu-h,\quad \bar{p}_\mu=p_\mu-p,\quad \bar{\rho}_\mu = \rho_\mu -\rho\,.
$$
Direct calculation shows $(\bar{u}_\mu, \bar{h}_\mu, \bar{\rho}_\mu)$ satisfies
\begin{equation}\label{eq4.1}
\left\{\begin{array}{l}
\partial_t \bar{u}_\mu+(\bar{u}_\mu+u) \cdot \nabla \bar{u}_\mu -\mu\Delta(\bar{u}_\mu+u)= - \bar{u}_\mu\cdot \nabla u +\bar{h}_\mu\cdot\nabla h+h_\mu \cdot \nabla \bar{h}_\mu -\nabla \bar{p}_\mu+\bar{\rho}_\mu e_3, \\[2mm]
\partial_t \bar{h}_\mu+u_\mu \cdot \nabla \bar{h}_\mu-\Delta \bar{h}_\mu=-\bar{u}_\mu\cdot\nabla h +\bar{h}_\mu \cdot \nabla u+h_\mu\cdot\nabla\bar{u}_\mu, \\[2mm]
\partial_t \bar{\rho}_\mu+u_\mu \cdot \nabla \bar{\rho}_\mu-\Delta \bar{\rho}_\mu=-\bar{u}_\mu\cdot\nabla\rho , \\[2mm]
\nabla \cdot \bar{u}_\mu=\nabla \cdot \bar{h}_\mu=0.
\end{array}\right.
\end{equation}
By a standard $L^2$-energy method, combined with $\nabla\cdot\bar{u}_\mu=\nabla\cdot\bar{h}_\mu=0 $, we deduce from $(\ref{eq4.1})_1$ that
$$
\begin{aligned}
&\frac{1}{2} \frac{d}{d t}\left\|\bar{u}_\mu(t)\right\|_{L^2}^2+\mu\int_{\mathbb{R}_3}|\nabla\bar{u}_\mu(t,x)|^2dx\\
=& \mu \int_{\mathbb{R}^3} \Delta u\cdot \bar{u}_\mu dx-\int_{\mathbb{R}^3}\left(\bar{u}_\mu \cdot \nabla u\right) \cdot \bar{u}_\mu dx+\int_{\mathbb{R}^3}\left(\bar{h}_\mu \cdot \nabla h\right) \cdot \bar{u}_\mu dx \\
&-\int_{\mathbb{R}^3}\left(h_\mu \cdot \nabla \bar{u}_\mu\right) \cdot \bar{h}_\mu d x
+\int_{\mathbb{R}^3}\left(\bar{\rho}_\mu \cdot \bar{u}_\mu\right)dx.
\end{aligned}
$$
Using the same method for $\bar{h}_\mu$ and $\bar{\rho}_\mu$, we can get
$$
\begin{aligned}
\frac{1}{2} \frac{d}{d t}\left\|\bar{h}_\mu(t)\right\|_{L^2}^2+\int_{\mathbb{R}^3}\left|\nabla \bar{h}_\mu(t, x)\right|^2 d x= & -\int_{\mathbb{R}^3}\left(\bar{u}_\mu \cdot \nabla h\right) \cdot \bar{h}_\mu d x+\int_{\mathbb{R}^3}\left(\bar{h}_\mu \cdot \nabla u\right) \cdot \bar{h}_\mu d x \\
& +\int_{\mathbb{R}^3}\left(h_\mu \cdot \nabla \bar{u}_\mu\right) \cdot \bar{h}_\mu dx ,
\end{aligned}
$$
and
$$
\begin{aligned}
\frac{1}{2} \frac{d}{d t}\left\|\bar{\rho}_\mu(t)\right\|_{L^2}^2+\int_{\mathbb{R}^3}\left|\nabla \bar{\rho}_\mu(t, x)\right|^2 d x= & -\int_{\mathbb{R}^3}\left(\bar{u}_\mu \cdot \nabla \rho\right) \cdot \bar{\rho}_\mu d x.
\end{aligned}
$$
Combining with the above estimates, and through 
$$
\mu\int_{\mathbb{R}^3}|\nabla \bar{u}_\mu(t,x)|^2dx\geq 0, \int_{\mathbb{R}^3}|\nabla \bar{h}_\mu(t,x)|^2dx\geq 0, \int_{\mathbb{R}^3}|\nabla \bar{\rho}_\mu(t,x)|^2dx\geq 0,
$$
we derive
$$
\begin{aligned}
\frac{1}{2} \frac{d}{d t}\left(\left\|\bar{u}_\mu(t)\right\|_{L^2}^2+\left\|\bar{h}_\mu(t)\right\|_{L^2}^2+\left\|\bar{\rho}_\mu(t)\right\|_{L^2}^2\right)\leq& \mu \int_{\mathbb{R}^3} \Delta u \cdot \bar{u}_\mu dx\\
& -\int_{\mathbb{R}^3}\left(\bar{u}_\mu \cdot \nabla u\right) \cdot \bar{u}_\mu \mathrm{d} x+\int_{\mathbb{R}^3}\left(\bar{h}_\mu \cdot \nabla h\right) \cdot \bar{u}_\mu dx \\
& -\int_{\mathbb{R}^3}\left(\bar{u}_\mu \cdot \nabla h\right) \cdot \bar{h}_\mu \mathrm{d} x+\int_{\mathbb{R}^3}\left(\bar{h}_\mu \cdot \nabla u\right) \cdot \bar{h}_\mu d x \\
& -\int_{\mathbb{R}^3}\left(\bar{u}_\mu \cdot \nabla \rho\right) \cdot \bar{\rho}_\mu \mathrm{d} x+\int_{\mathbb{R}^3}\left(\bar{\rho}_\mu \cdot \bar{u}_\mu\right)d x.
\end{aligned}
$$
Using the Cauchy-Schwartz inequality, it knows that
$$
\begin{aligned}
\frac{1}{2} \frac{d}{d t}\left(\left\|\bar{u}_\mu(t)\right\|_{L^2}^2+\left\|\bar{h}_\mu(t)\right\|_{L^2}^2+\left\|\bar{\rho}_\mu(t)\right\|_{L^2}^2\right)\leq&  \left\|\bar{u}_\mu\right\|_{L^2}\left(\mu\left\|\Delta u\right\|_{L^2}\right)\\
&+\left\|\bar{u}_\mu\right\|_{L^2}\left\|\bar{u}_\mu \cdot \nabla u\right\|_{L^2}+\left\|\bar{u}_\mu\right\|_{L^2}\left\|\bar{h}_\mu \cdot \nabla h\right\|_{L^2} \\
& +\left\|\bar{h}_\mu\right\|_{L^2}\left\|\bar{u}_\mu \cdot \nabla h\right\|_{L^2}+\left\|\bar{h}_\mu\right\|_{L^2}\left\|\bar{h}_\mu \cdot \nabla u\right\|_{L^2}\\
& +\left\|\bar{\rho}_\mu\right\|_{L^2}\left\|\bar{u}_\mu \cdot \nabla \rho\right\|_{L^2}+\left\|\bar{\rho}_\mu\right\|_{L^2}\left\|\bar{u}_\mu \right\|_{L^2}.
\end{aligned}
$$
Now, we integrate the above inequality over time. Noticing that $(\bar{u}_\mu, \bar{h}_\mu, \bar{\rho}_\mu)$ has zero initial data, we derives
$$
\begin{aligned}
\left\|\bar{u}_\mu\right\|_{L_t^{\infty} L^2}^2+\left\|\bar{h}_\mu\right\|_{L_t^{\infty} L^2}^2+\left\|\bar{\rho}_\mu\right\|_{L_t^{\infty} L^2}^2 \lesssim&\,\,\,\,\,\,\left\|\bar{u}_\mu\right\|_{L_t^{\infty} L^2}\left\|\bar{u}_\mu \cdot \nabla u\right\|_{L_t^1 L^2}+\left\|\bar{u}_\mu\right\|_{L_t^{\infty} L^2}\left\|\bar{h}_\mu \cdot \nabla h\right\|_{L_t^1 L^2} \\
& +\left\|\bar{h}_\mu\right\|_{L_t^{\infty} L^2}\left\|\bar{u}_\mu \cdot \nabla h\right\|_{L_t^1 L^2} +\left\|\bar{h}_\mu\right\|_{L_t^{\infty} L^2}\left\|\bar{h}_\mu \cdot \nabla u\right\|_{L_t^1 L^2} \\
&+\left\|\bar{\rho}_\mu\right\|_{L_t^{\infty} L^2}\left\|\bar{u}_\mu \cdot \nabla \rho\right\|_{L_t^1 L^2} +\left\|\bar{\rho}_\mu\right\|_{L_t^{\infty} L^2}\left\|\bar{u}_\mu\right\|_{L_t^1 L^2}\\
&+\left\|\bar{u}_\mu\right\|_{L_t^{\infty} L^2}\left(\mu\left\|\Delta u\right\|_{L_t^1 L^2}\right).\\
\end{aligned}
$$
Using Young's inequality, it yields
\begin{equation}\label{eq4.2}
\begin{aligned}
\left\|\bar{u}_\mu\right\|_{L_t^{\infty} L^2}^2+\left\|\bar{h}_\mu\right\|_{L_t^{\infty} L^2}^2+\left\|\bar{\rho}_\mu\right\|_{L_t^{\infty} L^2}^2 \lesssim&\,\,\,\,\,\,\left\|\bar{u}_\mu \cdot \nabla u\right\|_{L_t^1 L^2}^2+\left\|\bar{h}_\mu \cdot \nabla h\right\|_{L_t^1 L^2}^2+\left\|\bar{u}_\mu \cdot \nabla h\right\|_{L_t^1 L^2}^2\\
&+\left\|\bar{h}_\mu \cdot \nabla u\right\|_{L_t^1 L^2}^2+\left\|\bar{u}_\mu \cdot \nabla \rho\right\|_{L_t^1 L^2}^2 +\left\|\bar{u}_\mu\right\|_{L_t^1 L^2}^2\\
&+\left(\mu\left\|\Delta u\right\|_{L_t^1 L^2}\right)^2.\\
\end{aligned}
\end{equation}
We know from Young's inequality that
\begin{equation}\label{eq4.3}
\begin{aligned}
&2\left(\left\|\bar{u}_\mu\right\|_{L_t^{\infty} L^2}\left\|\bar{h}_\mu\right\|_{L_t^{\infty} L^2}+\left\|\bar{u}_\mu\right\|_{L_t^{\infty} L^2}\left\|\bar{\rho}_\mu\right\|_{L_t^{\infty} L^2}+\left\|\bar{h}_\mu\right\|_{L_t^{\infty} L^2}\left\|\bar{\rho}_\mu\right\|_{L_t^{\infty} L^2} \right)\\[2mm]
& \leq 2\left(\left\|\bar{u}_\mu\right\|_{L_t^{\infty} L^2}^2+\left\|\bar{h}_\mu\right\|_{L_t^{\infty} L^2}^2+\left\|\bar{\rho}_\mu\right\|_{L_t^{\infty} L^2}^2 \right).
\end{aligned}
\end{equation}
Combining with (\ref{eq4.2}), (\ref{eq4.3}) and employing that $(a+b+c)^2=a^2+b^2+c^2+2ab+2ac+2bc$, we obtain
$$
\begin{aligned}
\left(\left\|\bar{u}_\mu\right\|_{L_t^{\infty} L^2}+\left\|\bar{h}_\mu\right\|_{L_t^{\infty} L^2}+\left\|\bar{\rho}_\mu\right\|_{L_t^{\infty} L^2} \right)^2\lesssim&\,\,\,\,\,\,\left\|\bar{u}_\mu \cdot \nabla u\right\|_{L_t^1 L^2}^2+\left\|\bar{h}_\mu \cdot \nabla h\right\|_{L_t^1 L^2}^2+\left\|\bar{u}_\mu \cdot \nabla h\right\|_{L_t^1 L^2}^2\\
&+\left\|\bar{h}_\mu \cdot \nabla u\right\|_{L_t^1 L^2}^2+\left\|\bar{u}_\mu \cdot \nabla \rho\right\|_{L_t^1 L^2}^2 +\left\|\bar{u}_\mu\right\|_{L_t^1 L^2}^2\\
&+\left(\mu\left\|\Delta u\right\|_{L_t^1 L^2}\right)^2.\\
\end{aligned}
$$
This indicates
\begin{equation}\label{eq4.4}
\begin{aligned}
\left\|\bar{u}_\mu\right\|_{L_t^{\infty} L^2}+\left\|\bar{h}_\mu\right\|_{L_t^{\infty} L^2}+\left\|\bar{\rho}_\mu\right\|_{L_t^{\infty} L^2}
\lesssim&\,\,\,\,\,\,\left\|\bar{u}_\mu \cdot \nabla u\right\|_{L_t^1 L^2}+\left\|\bar{h}_\mu \cdot \nabla h\right\|_{L_t^1 L^2}+\left\|\bar{u}_\mu \cdot \nabla h\right\|_{L_t^1 L^2}\\
&+\left\|\bar{h}_\mu \cdot \nabla u\right\|_{L_t^1 L^2}+\left\|\bar{u}_\mu \cdot \nabla \rho\right\|_{L_t^1 L^2} +\left\|\bar{u}_\mu\right\|_{L_t^1 L^2}\\
&+\mu\left\|\Delta u\right\|_{L_t^1 L^2}.\\
\end{aligned}
\end{equation}
Denoting
$$
\mathscr{L}(t):=\left\|\bar{u}_\mu\right\|_{L_t^{\infty} L^2}+\left\|\bar{h}_\mu\right\|_{L_t^{\infty}L^2}+\left\|\bar{\rho}_\mu\right\|_{L_t^{\infty} L^2}\,.
$$
Clearly,
\[
\mathscr{L}(0)=0.
\]
Then applying H\"older's inequality and Young's inequality for the right side of (\ref{eq4.4}), we get
$$
\begin{aligned}
\mathscr{L}(t)\lesssim\mu\int_0^t\|\Delta u(\tau)\|_{L^2}d\tau+\int_0^t\left(1+\|\nabla u(\tau)\|_{L^\infty}+\|\nabla h(\tau)\|_{L^\infty}+\|\nabla \rho(\tau)\|_{L^\infty}\right)\mathscr{L}(\tau)d\tau\,.
\end{aligned}
$$
The Gr\"onwall inequality leads to
\begin{equation}\label{eq4.5}
\mathscr{L}(t)\lesssim e^{\int_0^t\left(1+\|\nabla u(\tau)\|_{L^\infty}+\|\nabla h(\tau)\|_{L^\infty}+\|\nabla \rho(\tau)\|_{L^\infty}\right)d\tau}\left(\mu\int_0^t\|\Delta u(\tau)\|_{L^2}d\tau\right).
\end{equation}
Using Sobolev interpolation and H\"older's inequality, if $m\geq 3$, (\ref{eq4.5}) follows that
$$
\left\|\bar{u}_\mu\right\|_{L_t^{\infty}L^2}+\left\|\bar{h}_\mu\right\|_{L_t^{\infty}L^2}+\left\|\bar{\rho}_\mu\right\|_{L_t^{\infty}L^2}\leq (\mu t)\|u\|_{L_t^{\infty}H^m}.
$$
Recalling Theorem \ref{Th1.1}, we arrive at
$$
\left\|u_\mu-u\right\|_{L_t^{\infty}L^2}+\left\|h_\mu-h\right\|_{L_t^{\infty}L^2}+\left\|\rho_\mu-\rho\right\|_{L_t^{\infty}L^2}\leq (\mu t)\Phi_{4,3}(t).
$$
This completes the proof of Theorem \ref{Th1.2}.

\vspace{1cm} % 在作者信息前添加一些垂直空间

\noindent % 阻止这一行自动缩进
\hspace*{1em} % 开头空两格作者姓名 \\
ZHAOJUN XING\\
School of Mathematics and Statistics, Nanjing University of Information Science and Technology, Nanjing 210044, China\\
\hspace*{1em}
Email address: zhaojunxing@nuist.edu.cn

\end{document}